\documentclass[a4paper, 12pt]{article}

\usepackage[T1]{fontenc} 
\usepackage[utf8]{inputenc} 
\usepackage{amsmath} 
\usepackage{amsfonts} 
\usepackage{amssymb} 
\usepackage{amsthm} 
\usepackage{mathrsfs}

\usepackage{tikz}

\usepackage{hyperref}

\newcommand{\N}{\mathbf{N}} 
\newcommand{\R}{\mathbf{R}} 

\newcommand{\IHL}{U}
\newcommand{\RHL}{V}
\newcommand{\Prob}{\mathbf P}
\newcommand{\E}{\mathbf E}

\newcommand{\eps}{\varepsilon}
\newcommand{\dd}{\,\mathrm{d}}

\theoremstyle{plain} 
\newtheorem{theorem}{Theorem} 
\newtheorem{lemma}{Lemma} 

\theoremstyle{definition} 

\theoremstyle{remark} 

\begin{document}
\date{\today}
\title{Asymptotics of the persistence exponent of integrated fractional Brownian motion and fractionally integrated Brownian motion}
\author{Frank Aurzada\footnote{Technical University of Darmstadt, Schloßgartenstraße 7, 64289 Darmstadt, Germany. E-mail: aurzada@mathematik.tu-darmstadt.de, kilian@mathematik.tu-darmstadt.de} \and Martin Kilian\footnotemark[1]}
\maketitle

\begin{abstract}
We consider the persistence probability for the integrated fractional Brownian motion and the fractionally integrated Brownian motion with parameter $H,$ respectively. For the integrated fractional Brownian motion, we discuss a conjecture of Molchan and Khokhlov and determine the asymptotic behavior of the persistence exponent as $H\to 0$ and $H\to 1,$ which is in accordance with the conjecture. For the fractionally integrated Brownian motion, also called Riemann-Liouville process, we find the asymptotic behavior of the persistence exponent as $H\to 0$.
\end{abstract}


\bigskip


\section{Introduction and main results}
The area of persistence probabilities deals with properties of stochastic processes when they have long excursions, i.e., when they stay in some fixed subset of their image for an untypically long time. For real-valued processes, one usually considers the event that the process stays on a half-line. The simplest question is the persistence probability itself: For a self-similar process $(X_t)_{t\ge 0}$ one expects that
\begin{equation}\label{eq:pppolynomial}
\Prob{\left(X_t < 1 \,\forall t \in {[0,T]}\right)} = T^{-\theta +o(1)},\qquad T\to\infty,
\end{equation}
for some constant $\theta=\theta(X)\in(0,\infty),$ called {\it persistence exponent}, which is to be determined.

This type of problem originates in the theoretical physics literature, where the persistence exponent serves as a simple measure of how fast a complicated physical system returns from a disordered initial condition to its stationary state. The question has received quite some attention in recent years for various types of processes. We refer to \cite{majumdar} for an overview of the theoretical physics point of view and to \cite{aurzadasimon} for a survey of the mathematics literature.

The present paper deals with the persistence exponents of two related processes, namely the integrated fractional Brownian motion $I^H$ and the fractionally integrated Brownian motion $R^H,$ which we will define now.

For $H \in {(0,1)},$ let $B^H$ be a standard fractional Brownian motion (FBM), i.e., a centered Gaussian process with covariance
\begin{equation*}
\E[ B^H_t B^H_s ] = \frac{1}{2} \left( t^{2H} + s^{2H} - |t-s|^{2H} \right), \qquad t,s\geq 0.
\end{equation*}
The persistence exponent of FBM is known to be $\theta(B^H)=1-H,$ see \cite{molchan1999} (and \cite{aurzada2011,adgpp2017,agpp2018} for refinements). In the present paper, we deal with the (one-sided) integrated version of $B^H,$ which we call $I^H,$ i.e.,
\begin{equation*}
I_t^H:=\int_0^t B_s^H \dd s, \qquad t \ge 0.
\end{equation*}
The persistence exponent $\theta_I(H):=\theta(I_H)$ exists due to the fact that $I^H$ has nonnegative correlations. However, its value is unknown unless $H=1/2$. In this case, $B^{1/2}$ is a usual Brownian motion and $I^{1/2}$ is an integrated Brownian motion, and it could be shown via Markov techniques that $\theta_I(1/2)=1/4$ (cf. \cite{Goldman1971}, \cite{Sinai1992}, and \cite{IsozakiWatanabe}).

For the general case, Molchan and Khokhlov stated the following conjecture \cite{MolchanKhoklov2004}:
\begin{equation*}
\theta_I(H)=H (1-H).
\end{equation*}
This conjecture is surprising because of its symmetry, as it is clear that $B^H$ (and thus $I^H$) are very different processes for $H<1/2$ and for $H>1/2$. Further, in the sequence of papers \cite{MolchanKhoklov2004}, \cite{Molchan2008}, \cite{Molchan2012}, \cite{molchan2018}, the following properties of $\theta_I(H)$ were established: $\min(H,1-H)/2 \leq \theta_I(H)\leq \min(H,1-H)$ for all $H\in(0,1)$;  $\theta_I(1-H)\leq \theta_I(H)$ for $H<1/2$; and $\theta_I(H)\leq \max(1/4,\sqrt{(1-H^2)/12})$ for all $H \in (0,1)$.

The present paper determines the asymptotic behavior of $\theta_I(H)$ for $H\to 0$ and as $H\to 1$. This is our first main result. Here and elsewhere, $f(x)\sim g(x)$ stands for $\lim f(x)/g(x) = 1$.

\begin{theorem}\label{theo:molchankhoklov} The function $H\mapsto \theta_I(H)$ is continuous on $(0,1)$. Further, $\theta_I(H) \sim H$ as $H \to 0$ and $\theta_I(H) \sim 1-H$ as $H \to 1$.
\end{theorem}

The second result of our paper deals with fractionally integrated Brownian motion, also known as Riemann-Liouville processes. For $H>0,$ let
\begin{equation} \label{eq:defnrl}
R_t^H:=\int_0^t (t-s)^{H- \frac {1} {2}} \dd B_s, \qquad t \ge 0,
\end{equation}
be the Riemann-Liouville fractional integral of a Brownian motion $(B_t)_{t \ge 0}$. For $H=1/2,$ this is just the Brownian motion $(B_t)$. A Fubini argument shows that, for any integer $n\geq 0,$
\begin{equation*}
R^{n+1/2}_t = n! \int_0^t \ldots \int_0^{s_{n-1}} B_{s_n} \dd s_n \ldots \dd s_1 = n! \,(I^n B)_t,\qquad t\geq 0,
\end{equation*}
where $(I f)_t := \int_0^t f(s) \dd s$ is the simple integration operator. So, the process defined in \eqref{eq:defnrl} is indeed a fractionally integrated Brownian motion.

In the case $H \in (0,1),$ Riemann-Liouville processes are closely related to FBMs $B^H$ via the Mandelbrot-van Ness integral representation
\begin{align*}
\sigma_H B^H_t &= \int_0^t (t-s)^{H- \frac {1} {2}} \dd B_s + \int_{-\infty}^0 \left((t-s)^{H-\frac {1} {2}} - (-s)^{H-\frac {1} {2}}\right) \dd B_s 
\\&=: R^H_t + M^H_t,
\end{align*}
where $\sigma_H:=\Gamma(H+1/2) / \sqrt{2 H \sin(\pi H) \Gamma(2H)}$ (see e.g.\ \cite[Theorem~1.3.1]{Mishura2008}). Clearly, $R^H$ and $M^H$ are independent processes. 

The persistence exponent of the Riemann-Liouville process, $\theta_R(H):=\theta(R^H),$ exists due to the fact that the process has nonnegative correlations. In \cite{aurzadadereich}, it was shown that $\theta_R(H)$ is nonincreasing.

For $H\to\infty,$ the correlation function of the Lamperti transform of $R^H$ (see below for precise definitions) converges to the correlation function  $\tau\mapsto 1/\cosh(\tau/2),$ the corresponding process having persistence exponent $\theta_R(\infty)\in(0,\infty)$.  Now, a continuity theorem for persistence exponents (see \cite[Theorem~1.6]{DemboMukherjee1}, \cite[Lemma~3.1]{dembomukherjee2}, or \cite[Lemma~3.6]{aurzadamukherjee}; these results are summarized in Lemma~\ref{lem:continuitypersistence} below in a way suitable for our purposes) shows that $\theta_R(H)\to \theta_R(\infty)$ as $H\to \infty$. In \cite{poplavskischehr}, the exponent is derived to be $\theta_R(\infty)=3/16$.

The present paper is concerned with the behavior of $\theta_R(H)$ as $H\to 0$. We show that $\theta_R(H)$ tends to infinity as $H\to 0$ and that the asymptotic behavior is in the range $H^{-1}$ to $H^{-2}$. This is our second main result. 

\begin{theorem}\label{theo:riemannliouville} The function $H\mapsto \theta_R(H)$ is continuous on $(0,\infty)$. Further, 
\begin{itemize}
\item[(a)] $\liminf_{H\to 0} \theta_R(H) H> 0$ and
\item[(b)] $\theta_R(H) H^2 \le 14^2$ for $H \in (0,1/2)$.
\end{itemize}
\end{theorem}

Figure~\ref{fig:plot} illustrates the behavior of the persistence exponents of Rie\-mann-Liouville process, (integrated) fractional Brownian motion, and (integrated) Brownian motion.

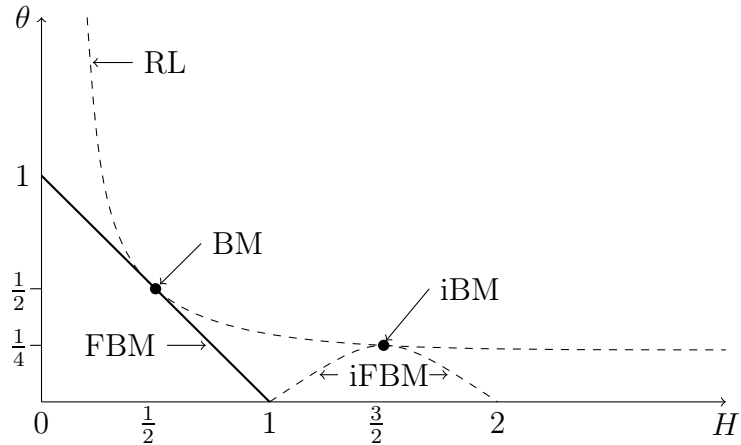
\begin{figure}
\centering
\begin{tikzpicture}[scale=3]
\draw[->] (0,0) -- (3,0) node[anchor=north] {$H$};
\draw[->] (0,0) -- (0,1.7) node[anchor=east] {$\theta$};
\draw	(0,0) node[anchor=north] {0}
		(1,0) node[anchor=north] {1}
		(2,0) node[anchor=north] {2}
		(0,1) node[anchor=east] {1}
		(-0.1,0.62) node[anchor=north] {$\frac{1}{2}$}
		(-0.05,0.5) -- (0,0.5)		
		(-0.1,0.37) node[anchor=north] {$\frac{1}{4}$}
		(-0.05,0.25) -- (0,0.25)
		(0.555,-0.1) node[anchor=east] {$\frac{1}{2}$}
		(1.555,-0.1) node[anchor=east] {$\frac{3}{2}$};

\draw [dashed] (0.2,1.7) .. controls (0.3,0.23) .. (3,0.23);
\draw[<-] (0.225,1.5) -- (0.4,1.5) node[anchor=west] {RL};

\draw [dashed] (1,0) .. controls (1.5,0.333) .. (2,0);
\draw[<-] (1.51,0.27) -- (1.7,0.5) node[anchor=west] {iBM};

\node at (0.5,0.5)[circle,fill,inner sep=1.5pt]{};
\draw[<-]  (0.52,0.52) -- (0.7,0.7) node[anchor=west] {BM};

\node at (1.5,0.25)[circle,fill,inner sep=1.5pt]{};
\draw[<-]  (1.22,0.120) -- (1.30,0.120) node[anchor=west] {iFBM};
\draw[->]  (1.70,0.120) -- (1.78,0.120);


\draw[thick] (0,1) -- (1,0);
\draw[<-]  (0.72,0.250) --  (0.550,0.25) node[anchor=east] {FBM};
\end{tikzpicture}
\caption{Relation of the persistence exponents of Riemann-Liouville process (RL), fractional Brownian motion (FBM), Brownian motion (BM), integrated Brownian motion (iBM), and integrated FBM (iFBM). For iFBM with parameter $H\in(0,1),$ we shift the function by $1$ because $H$-iFBM corresponds to $(H+1)$-RL.}
\label{fig:plot}
\end{figure}

The study of the persistence probabilities of FBM, iFBM, and related processes has received considerable attention in theoretical physics and mathematics. For instance, see \cite{MolchanKhoklov2004} and \cite{molchan2017}, where the Hausdorff dimension of Lagrangian regular points for the inviscid Burgers equation with FBM initial velocity is related to the two-sided persistence problem of integrated FBM; the interest for it arises from \cite{SheAurellFrisch1992} and \cite{Sinai1992}.

The rest of the paper is organized as follows. In the next subsection, we transform the problem for $I^H$ and $R^H$ into a persistence problem for Gaussian stationary processes (GSP) and sketch the proof technique. Section~\ref{sec:proofsifbm} contains the proofs related to Theorem~\ref{theo:molchankhoklov}, while Section~\ref{sec:proofsrl} is devoted to the proofs related to Theorem~\ref{theo:riemannliouville}.

\subsection*{Tranformation to GSP; ideas of the proofs}

The first step in our proofs is to transform the involved self-similar processes  -- $I^H$ and $R^H,$ which are $(H+1)$-self-similar and $H$-self-similar, respectively -- into GSPs via an exponential time change, also called Lamperti transform.

More generally, for an $H$-self-similar process $(X_t)_{t\ge 0},$ we consider its Lamperti tranform $Z_\tau:=e^{-H \tau} X_{e^{\tau}}, \ \tau\in\R$. It can often be shown that \eqref{eq:pppolynomial} turns into
\begin{equation} \label{eq:ppexponential}
\Prob{\left(Z_\tau < 0 \,\forall \tau \in {[0,T]}\right)} = e^{-T(\theta +o(1))},\qquad T\to\infty,
\end{equation}
where $\theta=\theta(X)$ is the same exponent as in \eqref{eq:pppolynomial}.

Consequently, we consider the Lamperti transform of $I^H$ defined by
\begin{equation*}
\IHL_{\tau}^H := \sqrt{2(1+H)} \,e^{-(1+H) \tau} \,I_{e^{\tau}}^H, \qquad \tau \in \R,
\end{equation*}
where the normalization constant is given in order to have a unit variance process.

Similarly, we consider the normalized Lamperti transform of $R^H$ defined by
\begin{equation*}
\RHL_{\tau}^H:=\sqrt{2H} \,e^{-\tau H} R_{e^\tau}^H, \qquad \tau \in \R.
\end{equation*}

The basic idea of our proofs is as follows.  The first step is to show that \eqref{eq:pppolynomial}  for $I^H$ ($R^H,$ respectively) is indeed the same as \eqref{eq:ppexponential} for $\IHL^H$ ($\RHL^H,$ respectively). This is a standard argument, where we follow \cite[Proposition~1.6]{aurzadadereich} or \cite[Theorem~1]{Molchan2008}.

The second step to prove Theorem~\ref{theo:molchankhoklov} (Theorem~\ref{theo:riemannliouville} is proved similarly, but the argument is much more technical) is to consider the GSP $(\IHL_{\tau/H}^H)_{\tau \in \R}$ as $H\to 0$ and the GSP  $(\IHL_{\tau/(1-H)}^H)_{\tau \in \R}$ as $H\to 1$. Their persistence exponents are given by $\theta_I(H)/H$ and $\theta_I(H)/(1-H),$ respectively, as a quick computation shows. We will show that in both of these cases, the respective correlation function of that GSP tends to the correlation function $\tau\mapsto e^{-\tau},$ which is the correlation function of an Ornstein-Uhlenbeck process, which has persistence exponent $1$. Then, we use the following lemma, which is Lemma~3.6 in \cite{aurzadamukherjee} together with Remark~3.8 in \cite{aurzadamukherjee} and Theorem~1.6 in \cite{DemboMukherjee1} as well as Lemma~3.10 in \cite{aurzadamukherjee}, to conclude the convergence of the persistence exponents $\theta_I(H)/H\to 1$ as $H\to 0$ and, respectively, $\theta_I(H)/(1-H)\to 1$ as $H\to 1$. 
\begin{lemma}\label{lem:continuitypersistence}
For $k \in \N,$ let $(Z_\tau^{(k)})_{\tau \ge 0}$ be a centered GSP with nonnegative correlation function $A_k(\tau), \,\tau \ge 0,$  satisfying $A_k(0)=1$. Suppose that $A_k(\tau) \to A(\tau)$ for $k \to \infty$ and all $\tau \ge 0,$ where $A\colon [0,\infty) \to [0,1]$ is the correlation function of a centered GSP $(Z_\tau)_{\tau \ge 0}$. 
\begin{itemize}
\item[(a)] If $Z^{(k)}$ and $Z$ have continuous sample paths and the conditions
\begin{align}
&\lim_{L \to \infty} \limsup_{k \to \infty} \sum_{\tau=L}^{\infty} A_k{\left(\frac {\tau} {\ell}\right)}=0 \text{ for every } \ell \in \N, \label{eq:18}
\\&\limsup_{\eps \downarrow 0} \left|\ln \eps\right|^{\eta} \sup_{k \in \N, \,\tau \in [0,\eps]} \left(1-A_k(\tau)\right) < \infty \text{ for some } \eta > 1, \label{eq:19}
\\&\limsup_{\tau \to \infty} \frac {\ln A(\tau)} {\ln \tau} < -1 \label{eq:20}
\end{align}
are fulfilled, then
\begin{align}\label{eq:convergenceofpes}
\lim_{k, T \to \infty} \frac {1}{T} \ln \Prob{\left(Z_{\tau}^{(k)} < 0 \,\forall \tau \in {[0,T]}\right)} = \lim_{T \to \infty} \frac {1}{T} \ln \Prob{\left(Z_{\tau} < 0 \,\forall \tau \in {[0,T]}\right)}. 
\end{align}
\item[(b)] If $A(\tau)=0$ for all $\tau > 0$ and \eqref{eq:18} is fulfilled, then
\begin{equation*}
-\lim_{k, T \to \infty} \frac {1}{T} \ln \Prob{\left(Z_{\tau}^{(k)} < 0 \,\forall \tau \in {[0,T]}\right)} = \infty.
\end{equation*}
\end{itemize}
\end{lemma}

The lemma says that if the correlation functions of the processes $Z^{(k)}$ converge pointwise to the correlation function of the process $Z$ and the technical conditions \eqref{eq:18}--\eqref{eq:20} are satisfied, then the persistence exponents of the  processes $Z^{(k)}$ converge to the persistence exponent of the process $Z$. Here, the existence of the persistence exponents, i.e., the existence of the (negative) limits in \eqref{eq:convergenceofpes}, follows from nonnegative correlations, Slepian's lemma, and subadditivity.

\section{Proofs for the case of integrated FBM}\label{sec:proofsifbm}
In this section, we prove Theorem \ref{theo:molchankhoklov}. We start with a lemma giving important properties of the correlation function of $\IHL^H$. In \cite[Lemma~2]{Molchan2008}, the correlation function $\rho_H(\tau)=\E{\left[\IHL_0^H \IHL_{\tau}^H\right]}$ was found to be
\begin{equation*}
\rho_H(\tau)=\frac {(1+H) \left(e^{-H \tau}+ e^{H \tau}\right)} {1+2H} + \frac {\left(e^{\tau/2} - e^{-\tau/2}\right)^{2 \,(1+H)}} {2 \,(1+2H)} - \frac {e^{(1+H) \tau} + e^{-(1+H) \tau}} {2 \,(1+2H)}
\end{equation*}
and shown to be nonincreasing on $(0,\infty).$ We show the following asymptotics. 

\begin{lemma}\label{lem:correlationX}
For all $\tau \ge 0,$
\begin{equation*}
\lim_{H \to 0} \rho_H{\left(\frac {\tau} {H}\right)} = \lim_{H \to 1} \rho_H{\left(\frac {\tau} {1-H}\right)} = e^{-\tau}.
\end{equation*}
\end{lemma}
\begin{proof}
For $H \to 0,$ we have
\begin{equation*}
\rho_H{\left(\frac {\tau} {H}\right)}=e^{-\tau} + e^{\tau} + \frac {e^{(1+H) \tau/H} \left((1 - e^{-\tau/H})^{2 \,(1+H)} - 1\right)} {2 \,(1+2H)} + o(1) \to e^{-\tau},
\end{equation*}
where we have used $(1-x)^a - 1 \sim -ax$ as $x \to 0$ uniformly on compact intervals in $a>0$. For $H \to 1,$ we have
\begin{align*}
\rho_H{\left(\frac {\tau} {1-H}\right)}&=\frac {(1+H) \,e^{H \tau/(1-H)}} {1+2H} 
\\&\qquad + \frac {e^{(1+H) \tau/(1-H)} \left((1-e^{-\tau/(1-H)})^{2 \,(1+H)} - 1\right)} {2 \,(1+2H)} + o(1)
\\&=\frac {(1+H) e^{-\tau}} {2} + o(1) \to e^{-\tau},
\end{align*}
where we have used $(1-x)^a - 1 + ax \sim a(a-1)x^2/2$ as $x \to 0$ uniformly on compact intervals in $a>0$.
\end{proof}

\begin{proof}[Proof of Theorem \ref{theo:molchankhoklov}]
Due to subadditivity, Slepian's lemma and the fact that $\IHL^H$ is a centered GSP with nonnegative correlations, the persistence exponent
\begin{equation}\label{eq:persistenceExp}
-\lim_{T \to \infty} \frac{1}{T} \ln \Prob{\left(\IHL_{\tau}^H < 0 \,\forall \tau \in {[0,T]}\right)}
\end{equation}
exists. Further, \eqref{eq:persistenceExp} equals $\theta_I(H)$. Indeed, note that
\begin{equation*}
\Prob{\left(\IHL_{\tau}^H < 0 \,\forall \tau \in {[0,T]}\right)} = \Prob{\left(I_t^H < 0 \,\forall t \in {[1,e^T]}\right)}
\end{equation*}
and that $\Prob{\left(I_t^H < 0 \,\forall t \in {[1,T]}\right)}$ has the same polynomial rate for $T \to \infty$ as $\Prob{\left(I_t^H < 1 \,\forall t \in {[0,T]}\right)},$ see \cite[Theorem~1]{Molchan2008}. \medskip

{\it The case $H\to 0$.} Observe that
\begin{align}
\frac {\theta_I(H)} {H} &= -\frac {1} {H} \lim_{T \to \infty} \frac {1}{T}\ln \Prob{\left(\IHL_{\tau}^H < 0 \,\forall \tau \in {[0,T]}\right)} \notag
\\&=-\frac {1} {H} \lim_{T \to \infty} \frac {1}{T/H} \ln \Prob{\left(\IHL_{\tau}^H < 0 \,\forall \tau \in {\left[0,\frac {T} {H}\right]}\right)} \notag
\\&=-\lim_{T \to \infty} \frac {1} {T} \ln \Prob{\left(\IHL_{\tau/H}^H < 0 \,\forall \tau \in {[0,T]}\right)}. \label{eq:thetabyH}
\end{align}
By Lemma \ref{lem:correlationX}, the correlation function $\tau \mapsto \rho_H{\left(\tau/H\right)}$ of $(\IHL_{\tau/H}^H)_{\tau\in\R}$ converges pointwise as $H \to 0$ to $\tau \mapsto e^{-\tau}$. This is the correlation function of a (scaled) Ornstein-Uhlenbeck process: If $(B_t)$ is a Brownian motion, then the process $e^{-\tau} B_{e^{2 \tau}}, \ \tau \in \R,$ which is the Lamperti transform of $(B_{t^2}),$ is a centered GSP with correlation function $\tau \mapsto e^{-\tau}$. Since $(B_t)$ has persistence exponent $1/2,$ the process $(B_{t^2})$ has persistence exponent $1$ and, by applying again \cite[Theorem~1]{Molchan2008}, this also holds for the corresponding Lamperti transform. 

So, as soon as we have also proved that the persistence exponents converge, the desired convergence $\theta_I(H)/H \to 1$ as $H \to 0$ follows. In order to achieve this, we want to apply Lemma \ref{lem:continuitypersistence}(a), i.e., we check the conditions \eqref{eq:18}--\eqref{eq:20} for the process $(\IHL_{\tau/H}^H)_{\tau \in \R}$ with the correlation function $\tau \mapsto \rho_H{\left(\tau/H\right)}$. Obviously, \eqref{eq:20} is fulfilled for the limiting correlation function $\tau \mapsto e^{-\tau}$.

To check \eqref{eq:18}, we write, by the binomial theorem,
\begin{align*}
\rho_H{\left(\frac {\tau} {H}\right)} &= \frac {(1+H) \,e^{-\tau}} {1+2H} 
\\&\qquad + \frac {e^{(1+H) \tau/H}} {2 \,(1+2H)} \sum_{k=2}^{\infty} (-1)^{k} \binom{2+2H} {k} e^{-k \tau/H} - \frac {e^{-(1+H) \tau/H}} {2 \,(1+2H)}
\\&\le \frac {(1+H) \,e^{-\tau}} {1+2H} + \frac {(1+H) \,e^{-(1-H) \tau/H}} {2 \,(1+2H)} \le \frac {3 \,e^{-\tau}} {2}
\end{align*}
for every $H < 1/2,$ since then it is easy to see that $(-1)^{k} \binom{2+2H} {k} \le 0$ for every $k \ge 3$.

Similarly for (5) we write, for every $H < 1/2,$ using the nonincreasing character of $\rho_H$ and the fact that $1-e^{-x}\leq x,$
\begin{align}
&\sup_{\tau \in [0,\eps]} \left(1-\rho_H{\left(\frac {\tau} {H}\right)}\right) = 1-\rho_H{\left(\frac {\eps} {H}\right)}=\rho_H{\left(\frac {0} {H}\right)}-\rho_H{\left(\frac {\eps} {H}\right)} \notag
\\&=\frac {(1+H) \left(1-e^{-\eps}\right)} {1+2H}  - \frac {1-e^{-(1+H) \eps/H}} {2 \,(1+2H)} \notag
\\&\qquad + \frac {\sum_{k=2}^{\infty} (-1)^{k} \binom{2+2H} {k} \left(1-e^{\eps (1+H-k)/H}\right)} {2 \,(1+2H)} \notag
\\&\le \eps +  \left(\frac {1+H} {2} - \frac {1} {2 (1+2H)}\right) \left(1-e^{-\eps/H}\right) \le \frac {5 \,\eps} {2},  \label{eq:check19iFBM}
\end{align}
for all $\eps > 0$. Note that in this case, the binomial theorem also holds for $\eps=0$ due to the fact that $2+2H > 0$ (see, e.g., \cite{Abel}). \medskip

{\it The case $H\to 1$.} Similarly to \eqref{eq:thetabyH},
\begin{equation*}
\frac {\theta_I(H)} {1-H} =-\lim_{T \to \infty} \frac {1} {T} \ln \Prob{\left(\IHL_{\frac {\tau} {1-H}}^H < 0 \,\forall \tau \in {[0,T]}\right)},
\end{equation*}
and by Lemma \ref{lem:correlationX}, the correlation function $\tau \mapsto \rho_H{\left(\tau/(1-H)\right)}$ of $(\IHL_{\tau/(1-H)}^H)$ converges pointwise as $H \to 1$ to $\tau \mapsto e^{-\tau}$. Again, this is the correlation function of the Ornstein-Uhlenbeck process with persistence exponent $1$. Applying Lemma \ref{lem:continuitypersistence}(a) for the process $(\IHL_{\tau/(1-H)}^H)_{\tau\in\R}$ completes the proof of the asymptotics, subject to checking the technical conditions. We have already seen that \eqref{eq:20} is fulfilled, since the limiting correlation function is the same as in the $H \to 0$ case. 

Next, we check condition \eqref{eq:18}. For $H > 1/2,$ we see that $(-1)^k \binom{2+2H} {k} < 0$ for $k=3$ and $(-1)^k \binom{2+2H} {k} \ge 0$ for $k \ge 4$. So, estimating again the negative terms by $0,$ we get
\begin{align*}
\rho_H{\left(\frac {\tau} {1-H}\right)} &\le \frac {(1+H) \,e^{-H \tau/(1-H)}} {1+2H} + \frac {(1+H) \,e^{-\tau}} {2} 
\\&\qquad + \frac {1} {2 (1+2H)} \sum_{k=4}^{\infty} (-1)^{k} \binom{2+2H} {k} e^{(1+H-k) \tau/(1-H)}
\\&\le \left(2 + \frac {1} {2\,(1+2H)} \sum_{k=4}^{\infty} (-1)^{k} \binom{2+2H} {k}\right) e^{-\tau} 
\\&= \left(\frac {5} {2} + (1+H) \left(\frac {H} {3} - \frac {1} {2}\right)\right) e^{-\tau} \le \frac {13 \,e^{-\tau}} {6}
\end{align*}
for $H \in (1/2,1),$ where in the last equality, again by the binomial theorem, we used the fact that $\sum_{k=0}^{\infty} (-1)^{k} \binom{2+2H} {k} = (1-1)^{2+2H} = 0$.

Further, condition \eqref{eq:19} can be verified similarly to \eqref{eq:check19iFBM}, since in this case
\begin{align*}
&\sup_{\tau \in [0,\eps]} \left(1-\rho_H{\left(\frac {\tau} {1-H}\right)}\right) = \rho_H{\left(\frac {0} {1-H}\right)}-\rho_H{\left(\frac {\eps} {1-H}\right)}
\\&=\frac {(1+H) \left(1-e^{-\eps}\right)} {2} + \frac {(1+H) \left(1-e^{-H \eps/(1-H)}\right)} {1+2H} - \frac {1-e^{-(1+H) \eps/(1-H)}} {2 \,(1+2H)}
\\&\qquad + \frac {1} {2 \,(1+2H)} \sum_{k=3}^{\infty} (-1)^{k} \binom{2+2H} {k} \left(1-e^{\eps (1+H-k)/(1-H)}\right)
\\&\le \eps + \left(\frac {1+H} {1+2H} - \frac {H (1+H)} {3}\right) \frac {\eps} {1-H}
\\&\qquad + \frac {H (1+H) (2H-1)} {12} \left(1-e^{-(3-H) \eps/(1-H)}\right) 
\\&\qquad - \frac {1} {2 \,(1+2H)} \left(1-e^{-(1+H) \eps/(1-H)}\right) + \sum_{k=5}^{\infty} (-1)^{k} \binom{2+2H} {k} \frac {(k-1-H) \,\eps} {1-H},
\end{align*}
where we used again $1-e^{-x} \le x$ and estimated $\eps \,H/(1-H) \le \eps/(1-H)$ as well as $\eps \,(2-H)/(1-H) \ge \eps/(1-H)$ in the last step. Note that
\begin{equation*}
\frac {1+H} {1+2H} - \frac {H (1+H)} {3} = \frac {(3+2H) (1+H) (1-H)} {3 (1+2H)};
\end{equation*}
that
\begin{align*}
&\frac {H (1+H) (2H-1)} {12} \left(1-e^{-(3-H) \eps/(1-H)}\right) - \frac {1} {2 \,(1+2H)} \left(1-e^{-(1+H) \eps/(1-H)}\right)
\\&\le \left(\frac {H (2H-1) (3-H)} {12} - \frac {1} {2 \,(1+2H)}\right) \left(1-e^{-(1+H) \eps/(1-H)}\right)
\\&=\frac {H (4H^2 - 1) (3-H) - 6} {12 \,(1+2H)} \left(1-e^{-(1+H) \eps/(1-H)}\right) < 0,
\end{align*}
for $H \in (0,1)$ and $\eps > 0$ since $(1-x^{3-H})/(1-x^{1+H}) < (3-H)/(1+H)$ for $x \in (0,1);$ and that, for $k \ge 5,$
\begin{align*}
(-1)^k \binom{2+2H} {k} \frac {k-1-H} {1-H} &= \frac {2H+2} {k-2} \cdot \frac {(k-1-H) (2H+1)} {k-1} \cdot \frac {2H} {k} \cdot \frac {2H-1} {1} 
\\&\qquad \times \frac {2-2H} {2 (1-H)} \cdot \frac{3-2H}{3} \cdots \frac {k-3-2H} {k-3}
\\&\le \frac {4} {k-2} \cdot 3 \cdot \frac {2} {k} \cdot 1 \cdots 1 = \frac {12} {k (k-2)}, 
\end{align*}
which is summable in $k$. Putting these facts together, we get, for every $\eta>1,$
\begin{align*}
&\limsup_{\eps \to 0} \left|{\ln \eps}\right|^{\eta} \sup_{H \in [1/2,1), \tau \in [0,\eps]} \left(1-\rho_H{\left(\frac {\tau} {1-H}\right)}\right)
\\&\le \limsup_{\eps \to 0} \left|{\ln \eps}\right|^{\eta} \eps \sup_{H \in [1/2,1)} \left(1+\frac {(3+2H) (1+H)} {3 (1+2H)} + \sum_{k=5}^{\infty} \frac {12} {k (k-2)}\right) 
\\&= 0 < \infty,
\end{align*}
showing \eqref{eq:19}.

Finally, the continuity of $\theta_I$ follows from that of $H \mapsto \rho_H(\tau)$ and Lemma \ref{lem:continuitypersistence}(a), since it is easily seen that conditions \eqref{eq:18}--\eqref{eq:20} are satisfied for the sequence $\tau \mapsto \rho_H(\tau), \ H \in [H_0-\delta, H_0+\delta],$ with fixed $H_0 \in (0,1),$ small $\delta > 0$ and $H \to H_0$.
\end{proof}

\section{Proofs for the case of Riemann-Liouville processes} \label{sec:proofsrl}
In this section, we prove Theorem \ref{theo:riemannliouville}. For this purpose, we need the following two lemmas on the correlation function of $\RHL^H$.

The correlation function $r_H(\tau)=\E{\left[\RHL_0^H \RHL_\tau^H\right]}$ can be found, e.g., in \cite[(12)]{Lim2015} and reads
\begin{equation}\label{eq:correlationRL}
r_H(\tau)=\frac {4H} {1+2H} \,e^{-\tau/2} {}_2 F_1{\left(1,\frac {1} {2}-H,\frac {3} {2}+H,e^{-\tau}\right)}
\end{equation}
with the standard notation for the Gauss' hypergeometric function. From this expression, it is clear that $r_H(\tau)$ is decreasing on $(0,\infty)$ for every $H \in (0,1/2)$. We also have the following
representation. 
\begin{lemma}\label{lem:correlationX2constanttau}
For all $\tau,H > 0,$
\begin{equation*}
e^{-\tau/2}-r_H(\tau) = \frac {1-2H} {1+2H} \,e^{-\tau/2} \,(1-e^{-\tau})^{2H} \,{}_2 F_1{\left(\frac {1} {2}+H,2H,\frac {3} {2}+H,e^{-\tau}\right)}.
\end{equation*}
\end{lemma}
\begin{proof}
The result follows by applying Euler's transform of $_2 F_1$ and from the formula
\begin{equation*}
e^{-\tau/2}-r_H(\tau) = \frac {1-2H} {1+2H} \,e^{-\tau/2} \,(1-e^{-\tau}) \,{}_2 F_1{\left(1,\frac {3} {2}-H,\frac {3} {2}+H,e^{-\tau}\right)}.
\end{equation*}
To verify this formula note that after plugging in the definiton of $r_H,$ we are left with showing
\begin{align*}
&4H \,{}_2 F_1{\left(1,\frac {1} {2}-H,\frac {3} {2}+H,x\right)} 
\\&\qquad + (1-2H) (1-x) \,{}_2 F_1{\left(1,\frac {3} {2}-H,\frac {3} {2}+H,x\right)} = 1+2H
\end{align*}
for every $H>0$ and $x \in (0,1)$. But this contiguous relationship is easily obtained in equating the coefficients of $x^n$ in the two series. 
\end{proof}

Let us now analyze the behavior of the rescaled correlation $r_H(\tau/\gamma)$ with $\gamma=\gamma_H\to\infty$ as $H\to 0$.

\begin{lemma}\label{lem:correlationwithgamma}
Let $\gamma=\gamma_H$ be a function tending to infinity with $H\to 0$. If $\gamma^{-2H}\to c$ for $H\to 0$ and some $c \in [0,1],$ then $r_H(\tau/\gamma) \to 1-c$ for $H \to 0$ and all $\tau > 0$.
\end{lemma}
\begin{proof}
By Lemma \ref{lem:correlationX2constanttau},
\begin{equation*}
e^{-\tau/(2\gamma)} - r_H{\left(\frac {\tau} {\gamma}\right)} \sim (1-e^{-\tau/\gamma})^{2H} {}_2 F_1{\left(\frac {1} {2}+H,2H,\frac {3} {2}+H,e^{-\tau/\gamma}\right)}
\end{equation*}
as $H \to 0$. Now $(1-e^{-\tau/\gamma})^{2H} \sim (\tau/\gamma)^{2H} \to c$ as $H \to 0$ and
\begin{align*}
&\left|{}_2 F_1{\left(\frac {1} {2}+H,2H,\frac {3} {2}+H,e^{-\tau/\gamma}\right)} - 1\right|
\\&=\frac {\Gamma(3/2+H)} {\Gamma(1/2+H) \,\Gamma(2H)} \sum_{n=1}^\infty \frac{\Gamma(n+1/2+H) \,\Gamma(n+2H)} {\Gamma(n+3/2+H)} \cdot \frac {e^{-n \tau/\gamma}} {n!} 
\\&=\frac {1+2H} {2 \Gamma(2H)} \sum_{n=1}^{\infty} \frac {\Gamma(n+2H)} {(n-1)!} \cdot \frac {e^{-n \tau/\gamma}} {(n+1/2+H) n}
\\&\le \frac {1} {\Gamma(2H)} \sum_{n=1}^{\infty} n^{-2} \to 0, \qquad H \to 0. \qedhere
\end{align*}
\end{proof}

\begin{proof}[Proof of Theorem \ref{theo:riemannliouville}]
Similarly to the proof of Theorem \ref{theo:molchankhoklov}, due to nonnegative correlations, the persistence exponent
\begin{equation*}
-\lim_{T \to \infty} \frac{1}{T} \ln \Prob{\left(\RHL_{\tau}^H < 0 \,\forall \tau \in {[0,T]}\right)}
\end{equation*}
exists and equals $\theta_R(H)$. Again, \cite[Theorem~1]{Molchan2008} can be applied to see that $\Prob{\left(R_t^H < 0 \,\forall t \in {[1,T]}\right)}$ and $\Prob{\left(R_t^H < 1 \,\forall t \in {[0,T]}\right)}$ have the same polynomial rate as $T \to \infty$. Note that, according to \cite[Corollary~3.1]{aurzadadereich}, a proper choice for $\varphi_T$ in \cite[Theorem~1]{Molchan2008} is given by $\varphi_T:=\varphi \equiv 0$ on $[0,1/2)$ and
\begin{equation*}
\varphi(t):=c \int_{1/2}^t (t-s)^{H-1/2} s^{-\eta} \dd s, \qquad t \ge \frac {1} {2},
\end{equation*}
for $\eta \in (1/2,1/2+H)$ and $c$ large enough such that $\varphi(t) \ge 1$ for all $t \ge 1$. Note that
\begin{align*}
\varphi(t)&=c \,t^{H+1/2-\eta} \int_{1/2}^t \left(1-\frac {s} {t}\right)^{H-1/2} \left(\frac {s} {t}\right)^{-\eta} \frac{\dd s}{t}
\\&=c \,t^{H+1/2-\eta} \int_{1/(2t)}^1 (1-u)^{H-1/2} u^{-\eta} \dd u,
\end{align*}
which is an increasing function in $t$ so that 
\begin{equation*}
1/c=\int_{1/2}^1 (1-s)^{H-1/2} s^{-\eta} \dd s
\end{equation*}
is a suitable choice.

Now, similarly to \eqref{eq:thetabyH}, for every $\gamma,$ we have
\begin{equation*}
\frac {\theta_R(H)} {\gamma} =-\lim_{T \to \infty} \frac {1} {T} \ln \Prob{\left(\RHL_{\tau/\gamma}^H < 0 \,\forall \tau \in {[0,T]}\right)}.
\end{equation*}
We next show $\theta_R(H)/\gamma \to \infty$ for any function $\gamma=\gamma_H$ with $\gamma \ll H^{-1},$ where $f(x) \ll g(x)$ means $\lim f(x)/g(x) = 0$. This proves part (a). 

Let $\gamma=\gamma_H$ be a function satisfying $\gamma \to \infty$ and $\gamma \ll H^{-1}$ as $H \to 0$. Since $\lim_{H \to 0} \gamma^{-2H} \ge \lim_{H \to 0} H^{2H} = 1$ and $\lim_{H \to 0} \gamma^{-2H} \le \lim_{H \to 0} 1^{-2H} = 1,$ it follows from Lemma~\ref{lem:correlationwithgamma} that $r_H(\tau/\gamma) \to 0$ for $H \to 0$ and all $\tau > 0$. To conclude the assertion $\theta_R(H)/\gamma \to \infty,$ we want to apply Lemma \ref{lem:continuitypersistence}(b) and thus have to check \eqref{eq:18} for the correlation function $\tau \mapsto r_H(\tau/\gamma)$. Indeed, one has, by \eqref{eq:correlationRL} and the series representation of the hypergeometric function, for every $\ell \in \N,$
\begin{align*}
&\int_0^\infty r_H{\left(\frac {\tau} {\ell \gamma}\right)} \dd \tau=\ell \gamma \int_0^{\infty} r_H(\tau) \dd \tau
\\&= \ell \gamma \,\frac {4H \,\Gamma(3/2+H)} {(1+2H) \,\Gamma(1/2-H)} \sum_{k=0}^{\infty} \frac {\Gamma(1/2-H+k)} {\Gamma(3/2+H+k)} \int_0^{\infty} e^{-\tau (k+1/2)} \dd \tau \notag
\\&\sim 2 \ell \gamma H \sum_{k=0}^{\infty} \left(k+\frac {1} {2}\right)^{-2} =: c \cdot \ell \gamma H, \qquad H \to 0,
\end{align*}
and thus, for every $\ell, L \in \N,$
\begin{align*}
\limsup_{H \to 0} \sum_{\tau=L}^{\infty} r_H{\left(\frac {\tau} {\ell \gamma}\right)} &\le \limsup_{H \to 0} \int_{L-1}^{\infty} r_H{\left(\frac {\tau} {\ell \gamma}\right)} \dd \tau 
\\&\le \limsup_{H \to 0} \int_{0}^{\infty} r_H{\left(\frac {\tau} {\ell \gamma}\right)} \dd \tau = c \,\ell \,\limsup_{H \to 0} \gamma H = 0,
\end{align*}
where we used the monotonicity and nonnegativity of $r_H$. 

Now, we prove part (b). We will show that 
\begin{equation*}
r_H(\tau) \ge (1-|\tau|^H)_+ = \E[S_0^{H/2} S_\tau^{H/2}]
\end{equation*}
for $H \in (0,1/2)$ and all $\tau \in \R,$ where $(S_\tau^H)$ is the so-called fractional Slepian's process (see \cite[Section~2.3]{Molchan2012}). Then Slepian's lemma implies that $\theta_R(H) \le \theta_S(H/2),$ where $\theta_S(H)$ denotes the persistence exponent of $(S_\tau^H),$ and the assertion follows by \cite[Proposition~2.9]{Molchan2012}. 

We have
\begin{align*}
&\frac {1-2H} {1+2H} \,{}_2 F_1{\left(\frac {1} {2}+H,2H,\frac {3} {2}+H,e^{-\tau}\right)} \le \frac {1-2H} {1+2H} \,{}_2 F_1{\left(\frac {1} {2}+H,2H,\frac {3} {2}+H,1\right)} 
\\&= \frac {1-2H} {1+2H} \cdot \frac {\Gamma(3/2+H) \,\Gamma(1-2H)} {\Gamma(3/2-H)} = \frac {\Gamma(1/2+H) \,\Gamma(1-2H)} {\Gamma(1/2-H)} 
\\&= \frac {\Gamma(1/2+H)} {\Gamma(1/2)} \cdot \frac {\Gamma(1-H)} {2^{2H}} \le 1, \qquad H \in {\left(0,\frac {1} {2}\right)},
\end{align*}
where we used the Legendre duplication formula in the last equality and the monotonicity of $\Gamma(\cdot)$ on $(1/2,1)$ as well as the fact that
\begin{align*}
\Gamma(1-H)&=\Gamma{\left(2H \cdot \frac {1} {2} + (1-2H) \cdot 1\right)} 
\\&\le \left(\Gamma{\left(\frac {1} {2}\right)}\right)^{2H} \cdot \left(\Gamma(1)\right)^{1-2H} = \pi^H \le 2^{2H}, \qquad 2H \in (0,1),
\end{align*}
(due to the logarithmic convexity of $\Gamma(\cdot)$) in the last inequality. Together with Lemma \ref{lem:correlationX2constanttau}, this gives
\begin{equation*}
r_H(\tau) \ge e^{-\tau/2} \left(1-\left(1-e^{-\tau}\right)^{2H}\right) \ge e^{-\tau/2} \left(1-\tau^{2H}\right)=\phi(\tau) \left(1-\tau^H\right)
\end{equation*}
for $\tau \ge 0,$ where $\phi(\tau):=e^{-\tau/2} (1+\tau^H).$ Now, $\varphi(\tau):=1+\tau^H-e^{\tau/2}$ satisfies $\varphi(0)=0,$ $\varphi(1)=2-\sqrt{e} > 0$ and
\begin{equation*}
\varphi''(\tau)=-H (1-H) \tau^{-(2-H)} - \frac {e^{\tau/2}} {4} < 0, \qquad \tau \ge 0,
\end{equation*}
which implies $\varphi(\tau) \ge 0$ and thus $\phi(\tau) \ge 1$ for $\tau \in [0,1]$. This shows $r_H(\tau) \ge (1-|\tau|^H)_+$ for $\tau \in [0,1]$ and due to the symmetry and the nonnegativity of $r_H$ even for all $\tau \in \R$.  

Finally, similarly to the proof of Theorem \ref{theo:molchankhoklov}, the continuity of $\theta_R$ follows from the continuity of $H \mapsto r_H(\tau)$ and Lemma \ref{lem:continuitypersistence}(a), since the sequence $\tau \mapsto r_H(\tau), \ H \in [H_0-\delta, H_0+\delta],$ with fixed $H_0 \in (0,\infty),$ small $\delta > 0$ and $H \to H_0$ fulfills conditions \eqref{eq:18}--\eqref{eq:20}. One checks easily \eqref{eq:18} and \eqref{eq:20}, while for checking \eqref{eq:19}, we note that
\begin{align*}
1-r_H(\eps) = 1- e^{-\eps/2} + e^{-\eps/2} - r_H(\eps) \le \frac {\eps} {2} + c_{H_0} \eps^{2 (H_0-\delta)}
\end{align*}
for suitable $c_{H_0}$ and small $\eps$ using Lemma \ref{lem:correlationX2constanttau} with $\tau$ replaced by $\eps$ and the fact that $1-e^{-x}\leq x$. 
\end{proof}

\noindent\textbf{Acknowledgments.} This work was supported by Deutsche
Forschungsgemeinschaft (DFG grant AU370/5). We would like to thank the anonymous referee for his/her valuable suggestions, in particular for simplifying the proofs of Lemma \ref{lem:correlationX2constanttau} and Lemma \ref{lem:correlationwithgamma}. Further, we would like to thank G. Molchan for providing the proof of Theorem \ref{theo:riemannliouville}(b), which improves the result $\limsup_{H \to 0} \theta_R(H) H^2 < \infty$ in the first version of our paper and, more importantly, gives a significantly shorter proof of this result.  

\nocite{*}
\bibliographystyle{plain}

\begin{thebibliography}{10}

\bibitem{Abel}
N.~H. Abel.
\newblock Untersuchungen \"{u}ber die {R}eihe: $1+\frac{m}{1}x+\frac{m\ .\
  (m-1)}{1\ .\ 2}\ .\ x^2+\frac{m\ .\ (m-1)\ .\ (m-2)}{1\ .\ 2\ .\ 3}\ .\
  x^3+\ldots \text{u.s.w.}$.
\newblock {\em Journal für die reine und angewandte Mathematik}, 1:311--339,
  1826.

\bibitem{aurzada2011}
F.~Aurzada.
\newblock On the one-sided exit problem for fractional {B}rownian motion.
\newblock {\em Electronic Communications in Probability}, 16:392--404, 2011.

\bibitem{aurzadadereich}
F.~Aurzada and S.~Dereich.
\newblock {Universality of the asymptotics of the one-sided exit problem for
  integrated processes}.
\newblock {\em Annales de l'Institut Henri Poincar\'{e} Probabilit\'{e}s et
  Statistiques}, 49(1):236--251, 2013.

\bibitem{adgpp2017}
F.~Aurzada, A.~Devulder, N.~Guillotin-Plantard, and F.~P\`ene.
\newblock Random walks and branching processes in correlated {G}aussian
  environment.
\newblock {\em Journal of Statistical Physics}, 166(1):1--23, 2017.

\bibitem{agpp2018}
F.~Aurzada, N.~Guillotin-Plantard, and F.~P\`ene.
\newblock Persistence probabilities for stationary increment processes.
\newblock {\em Stochastic Processes and their Applications}, 128(5):1750--1771,
  2018.

\bibitem{aurzadamukherjee}
F.~Aurzada and S.~Mukherjee.
\newblock Persistence probabilities of weighted sums of stationary {G}aussian
  sequences.
\newblock Preprint, arXiv:2003.01192, 2020.

\bibitem{aurzadasimon}
F.~Aurzada and T.~Simon.
\newblock {Persistence probabilities and exponents}.
\newblock In {\em L\'{e}vy matters. {V}}, volume 2149 of {\em Lecture Notes in
  Mathematics}, pages 183--224. Springer, Cham, 2015.

\bibitem{majumdar}
A.~J. Bray, S.~N. Majumdar, and G.~Schehr.
\newblock {Persistence and first-passage properties in non-equilibrium
  systems}.
\newblock {\em Advances in Physics}, 62(3):225--361, 2013.

\bibitem{DemboMukherjee1}
A.~Dembo and S.~Mukherjee.
\newblock {No zero-crossings for random polynomials and the heat equation}.
\newblock {\em The Annals of Probability}, 43(1):85--118, 2015.

\bibitem{dembomukherjee2}
A.~Dembo and S.~Mukherjee.
\newblock Persistence of {G}aussian processes: non-summable correlations.
\newblock {\em Probability Theory and Related Fields}, 169(3-4):1007--1039,
  2017.

\bibitem{Goldman1971}
M.~Goldman.
\newblock {On the first passage of the integrated Wiener process}.
\newblock {\em The Annals of Mathematical Statistics}, 42(6):2150--2155, 1971.

\bibitem{IsozakiWatanabe}
Y.~Isozaki and S.~Watanabe.
\newblock {An asymptotic formula for the Kolmogorov diffusion and a refinement
  of Sinai's estimates for the integral of Brownian motion}.
\newblock {\em Proceedings of the Japan Academy, Series A, Mathematical
  Sciences}, 70(9):271--276, 1994.

\bibitem{Lim2015}
S.~C. Lim and C.~H. Eab.
\newblock {Some fractional and multifractional Gaussian processes: A brief
  introduction}.
\newblock {\em International Journal of Modern Physics: Conference Series},
  36:1560001, 2015.
  
\bibitem{Mishura2008}
Y.~S.~Mishura.
\newblock {\em {Stochastic Calculus for Fractional Brownian Motion and Related
  Processes}}, volume 1929 of {\em Lecture Notes in Mathematics}.
\newblock Springer, Berlin Heidelberg, 2008.

\bibitem{molchan1999}
G.~M.~Molchan.
\newblock Maximum of a fractional {B}rownian motion: probabilities of small
  values.
\newblock {\em Communications in Mathematical Physics}, 205(1):97--111, 1999.

\bibitem{Molchan2008}
G.~Molchan.
\newblock {Unilateral small deviations of processes related to the fractional
  Brownian motion}.
\newblock {\em Stochastic Processes and their Applications}, 118:2085--2097,
  2008.

\bibitem{Molchan2012}
G.~Molchan.
\newblock {Survival Exponents for Some Gaussian Processes}.
\newblock {\em International Journal of Stochastic Analysis}, 2012(137271),
  2012.

\bibitem{molchan2017}
G.~Molchan.
\newblock The inviscid {B}urgers equation with fractional {B}rownian initial
  data: the dimension of regular {L}agrangian points.
\newblock {\em Journal of Statistical Physics}, 167(6):1546--1554, 2017.

\bibitem{molchan2018}
G.~Molchan.
\newblock {Integrated fractional Brownian motion: persistence probabilities and
  their estimates}.
\newblock Preprint, arXiv:1806.04949, 2018.

\bibitem{MolchanKhoklov2004}
G.~Molchan and A.~Khokhlov.
\newblock {Small Values of the Maximum for the Integral of Fractional Brownian
  Motion}.
\newblock {\em Journal of Statistical Physics}, 114(3/4):923--945, 2004.

\bibitem{poplavskischehr}
M.~Poplavskyi and G.~Schehr.
\newblock {Exact persistence exponent for the 2d-diffusion equation and related
  Kac polynomials}.
\newblock {\em Physical Review Letters}, 121:150601, 2018.

\bibitem{SheAurellFrisch1992}
Z.-S. She, E.~Aurell, and U.~Frisch.
\newblock The inviscid {B}urgers equation with initial data of {B}rownian type.
\newblock {\em Communications in Mathematical Physics}, 148(3):623--641, 1992.

\bibitem{Sinai1992}
Ya.~G. Sinai.
\newblock {Distribution of some functionals of the integral of a random walk}.
\newblock {\em Theoretical and Mathematical Physics}, 90:219--241, 1992.

\end{thebibliography}

\end{document}